\documentclass{amsart}
\usepackage{amssymb, mathrsfs, bbm}
\usepackage[T1]{fontenc}
\usepackage[sc, osf]{mathpazo}
\selectfont
\usepackage[protrusion=true, expansion=true]{microtype}

\usepackage[all]{xy}
\SelectTips{eu}{}
\entrymodifiers={+!!<0pt,\fontdimen22\textfont2>}

\usepackage[pdftex, colorlinks=true, linkcolor=blue, citecolor=magenta, linktocpage]{hyperref}
\usepackage{rotfloat}
\usepackage{url}

\theoremstyle{plain}

\newtheorem{prop}[subsection]{Proposition}

\theoremstyle{remark}

\theoremstyle{definition}

\numberwithin{equation}{subsection}

\def\e{\mathbbm{1}}

\def\cA{\mathcal{A}}
\def\cB{\mathcal{B}}
\def\cC{\mathcal{C}}
\def\cD{\mathcal{D}}

\def\cO{\mathcal{O}}

\def\11{\mathbf{1}}

\def\BB{\mathbf{B}} 
\def\CC{\mathbf{C}}

\def\QQ{\mathbf{Q}} 
\def\RR{\mathbb{R}}

\def\ZZ{\mathbf{Z}}

 \def\fg{\mathfrak{g}}

 \def\fsl{\mathfrak{sl}}

 \def\can{\mathrm{can}}

 \def\dim{\mathrm{dim}}

 \def\Hom{\mathrm{Hom}}
 \def\id{\mathrm{id}}

 \def\pr{\mathrm{pr}}

 \def\var{\mathrm{var}}
 \def\wt{\mathrm{wt}}

 \newcommand{\smatrix}[1]{\left(\begin{smallmatrix}#1\end{smallmatrix}\right)}

 \newcommand{\mapright}[1]{\xrightarrow{#1}}

 \newcommand{\totimes}{\mathop{\widetilde{\otimes}}}

\makeatletter
\renewcommand{\@makefnmark}{\mbox{\textsuperscript{}}}
\makeatother

\title{Graded tensoring and crystals}
\author{R. Virk}
\email{rsvirk@gmail.com}
\begin{document}
\maketitle
\setcounter{tocdepth}{1}
\tableofcontents
\section{Introduction}
In the representation theory of reductive groups, semi-simple Lie algebras, etc., various categories of representations are often treated as `modules' for the corresponding monoidal category of finite dimensional representations (= Satake category). For instance, the Bernstein-Bernstein-Gelfand category $\cO$, of a semi-simple Lie algebra \cite{BGG}, is stable under tensoring with finite dimensional representations (and hence a module for the Satake category in a loose sense). Restricting to a block of $\cO$, corresponding to a fixed central character (say the trivial one), this Satake action induces an action of the Hecke category (for example, see \cite{St1}; also see \cite{So}). Of course, in order to get the Hecke action it is crucial to upgrade to a mixed setting (\'a la \cite{BGS}). 

The importance of the graded setting leads us to ask if it is possible to define an action of the Satake category on \emph{all} of graded category $\cO$ (not just a single block at a time)? Apart from intrinsic interest, this question is also motivated by trying to understand categories of Harish-Chandra modules (for real groups) for which category $\cO$ serves as a toy model. This line of questioning has been pursued by W. Soergel (unpublished). Soergel observes that although such `graded tensoring' may be defined, it is not associative. 
More precisely, for each finite dimensional representation $V$, one may define an endofunctor $V\totimes -$ on graded category $\cO$. Forgetting the grading yields the ordinary tensor product $V\otimes -$. Further, if $W$ is another finite dimensional representation, then 
\[ (V\otimes W)\totimes - \simeq V\totimes (W \totimes -).\]
However, these isomorphisms cannot be chosen in a coherent fashion (they do not satisfy the pentagon axiom). Put another way, the isomorphisms cannot be chosen to be compatible with their non-graded counterparts. 

In other words, one cannot obtain a monoidal functor from the Satake category (or even its quantum version) to the category of endofunctors of graded category $\cO$ (compatible with the usual tensor product). The underlying cause for this can be seen reasonably explicitly for $\fsl_2$: if such a functor were to exist, then graded tensoring with the tautological (2-dimensional) representation would be self bi-adjoint. Combinatorics at the Grothendieck group level yields that this is impossible (cf. \cite[Lemma 3.3]{BS}).

So, we have an action of some mysterious monoidal category on graded category $\cO$. From a distance it looks like the Satake category. Closer inspection reveals this to be a masquerade.

Due to lack of other evident candidates, Soergel has suggested that perhaps this is an action of the corresponding category of crystals. The purpose of this note is to provide some small evidence towards this (however, do see \cite{S}). I construct an action of the crystal category on (integral) graded category $\cO$ for $\fsl_2$ by brute force. Although I have not made it explicit, it is readily seen that this action lifts the usual non-graded tensor product (in an essentially unique way).

At the moment I do not have much of an understanding of the corresponding story for Harish-Chandra modules for $SL_2(\RR)$. I understand that forthcoming work of A. Glang and O. Straser will shed some light on this.

\textbf{Acknowledgments:} This work was done while I was visiting the Mathematisches Institut at Albert-Ludwigs-Universit\"at, Freiburg during the summer of 2012. I am indebted to W. Soergel for his hospitality and for patient tolerance of my daily barrage of silly questions. I am also grateful to A. Glang and O. Straser for so openly and patiently explaining me some of the ideas in their theses. Finally, my thanks to C. Stroppel for some very useful comments on an earlier version of this document.

\section{Conventions}
Categories will always be required to have zero objects. When dealing with monoidal categories the `$\otimes$' symbol will sometimes be omitted. For example, $FG$ will be written instead of $F\otimes G$, $F^2$ instead of $F\otimes F$, etc. Unit objects will be denoted by $\e$. An action of a monoidal category $\cC$ on a category $\cD$ will mean a monoidal functor from $\cC$ to endofunctors of $\cD$.

A $\QQ$-linear category will mean an additive category in which $n\cdot \id$ is an isomorphism for all $n\in \ZZ - \{0\}$. 
A $\QQ$-linear functor will mean a functor between $\QQ$-linear categories that respects the $\QQ$-vector space structure on $\Hom$-sets.

$\mathrm{Rep}(\fsl_2)$ will denote the monoidal category of finite dimensional representations of the Lie algebra $\fsl_2(\CC)$. The symbol `$L$' will be reserved for the tautological 2-dimensional irreducible representation in $\mathrm{Rep}(\fsl_2)$.

$\mathrm{Vect}^{\ZZ}$ will denote the category of graded (finite dimensional) $\CC$-vector spaces. For each $n\in \ZZ$ there is an isomorphism of categories 
\[ \langle n \rangle \colon \mathrm{Vect}^{\ZZ}\mapright{\sim} \mathrm{Vect}^{\ZZ} \]
(with inverse $\langle -n \rangle$) defined as follows: for a graded vector space $V$ with degree $i$ component $V_i$, the degree $i$ component of $V\langle n \rangle$ is $V_{i-n}$.

\section{Generators and relations}\label{s:gens}
Let $\cC'$ be the free monoidal category generated by one object $F$, morphisms $\e\to F^2$ and $F^2\to \e$, subject to the relations:
\begin{enumerate}
\item the composition $\e\to F^2 \to \e$ is the identity;
\item the composition $F \to F^2F = FF^2 \to F$ is zero;
\item the composition $F \to FF^2 = F^2F \to F$ is zero.
\end{enumerate}
A morphism $F^m \to F^n$ can be represented by the following pictorial data (string diagrams): 
\begin{itemize}
\item a closed rectangle in the plane with two opposite edges designated as the top and bottom row;
\item $n$ (resp. $m$) marked points on the top (resp. bottom) row;
\item smooth (pairwise) non-intersecting curves (`strings') in the rectangle connecting each point to exactly one other point.
\end{itemize}
Two string diagrams are considered equivalent if they induce the same pairing of the $n+m$ marked points.
Composition is accomplished by vertical juxtaposition, subject to the rules:
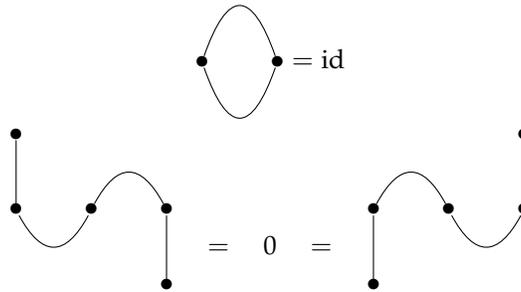
\[\xy 
(-5,0)*{\bullet};(5,0)*{\bullet}; 
**\crv{(0,-15)} 
**\crv{(0,15)}
\endxy
= \id
\]
\[\xy
(-5,15)*{\bullet}="t";
(-5,5)*{\bullet}="m"; 
(15,-5)*{\bullet}="b";
(5,5)*{\bullet}="m1";
(15,5)*{\bullet}="m2";
"b";"m2";**\dir{-};
"m";"m1";**\crv{(0,-5)};
"m";"t";**\dir{-};
"m1";"m2";**\crv{(10,15)};
\endxy
\quad=\quad 0\quad =\quad\xy
(15,15)*{\bullet}="t";
(-5,5)*{\bullet}="m"; 
(-5,-5)*{\bullet}="b";
(5,5)*{\bullet}="m1";
(15,5)*{\bullet}="m2";
"b";"m";**\dir{-};
"m2";"t";**\dir{-};
"m1";"m2";**\crv{(10,-5)};
"m";"m1";**\crv{(0,15)};
\endxy
\]
Tensor product is given by horizontal juxtaposition.
The translation from generators and relations to string diagrams is accomplished via the following prescription.
The identity map $F\to F$ is given by
\[ \xy
(0,5)*{\bullet}; (0,-5)*{\bullet};
**\dir{-}
\endxy
\]
The morphism $\e\to F^2$ is given by
\[\xy 
(-5,0)*{\bullet};(5,0)*{\bullet}; 
**\crv{(0,-10)} 
\endxy
\]
The morphism $F^2\to \e$ is given by
\[\xy 
(-5,0)*{\bullet};
(5,0)*{\bullet}; 
**\crv{(0,10)} 
\endxy
\]
\begin{prop}
We have
\[|\Hom_{\cC'}(F^m, F^n)-\{0\}| = \dim_{\CC} \Hom_{\mathrm{Rep}(\fsl_2)}(L^{\otimes m},L^{\otimes n}).\]
\end{prop}

\begin{proof}
There is a basis for $\Hom_{\mathrm{Rep}(\fsl_2)}(L^{\otimes m}, L^{\otimes n})$ that is in bijection with string diagrams representing distinct morphisms $F^m\to F^n$ (see \cite[\S2]{FK}).

Alternatively, the reader may find it pleasant to enumerate both the set of string diagrams and invariant vectors in $L^{\otimes m+n}$ using the Catalan numbers.
\end{proof}

Let $\cC$ denote the Karoubi envelope (sometimes also called the pseudo-abelian envelope) of the $\QQ$-linearization of the category $\cC'$. So $\cC$ comes equipped with a faithful functor $\cC' \to \cC$. Further, the monoidal structure on $\cC'$ extends to $\cC$, and upgrades $\cC'\to \cC$ to a monoidal functor. Note:
\[ \dim_{\QQ} \Hom_{\cC}(F^m, F^n) = \dim_{\CC} \Hom_{\mathrm{Rep}(\fsl_2)}(L^{\otimes m}, L^{\otimes n}).\]

Unfortunately, I do not know of a convenient reference that provides detailed constructions/proofs of all the relevant properties of Karoubi completions, linearizations, extensions of monoidal structures, etc. (much of it seems to be folklore in the literature on motives). Let me sketch some basic ideas to orient the reader.

To avoid making the language more cumbersome in what follows, the word `unique' (in reference to categories and functors) will be used instead of `unique up to equivalence', `unique up to natural isomorphism', etc.

The $\QQ$-linearization of $\cC'$ is a $\QQ$-linear category $\cC'\otimes \QQ$ equipped with a functor 
$\cC' \to \cC'\otimes \QQ$
satisfying the following universal property: if $\cC' \to \cA$ is a functor to a $\QQ$-linear category $\cA$, then there exists a unique functor $\cC'\otimes \QQ \to \cA$ such that the diagram
\[ \xymatrix{
\cC' \ar[r]\ar[d]& \cA \\
\cC'\otimes \QQ \ar@/^0pc/[ur]
}\]
commutes. If $\cC' \otimes \QQ$ exists, then it is unique. For an explicit construction, consider the category $\cA ux$ defined to have the same objects as $\cC'$, and morphisms given by
\[ \Hom_{\cA ux}(X,Y) = \QQ\mathrm{-span}\,(\Hom_{\cC'}(X,Y)-\{0\}). \]
Now $\cC'\otimes \QQ$ may be taken to be the completion of $\cA ux$ with respect to finite products.
The following is manifest from the construction.
\begin{prop}
The monoidal structure on $\cC'$ induces a monoidal structure on $\cC'\otimes \QQ$. The evident functor $\cC' \to \cC'\otimes \QQ$ is faithful and monoidal. If $\cC'\to \cA$ is a monoidal functor to a monoidal $\QQ$-linear category $\cA$, then the induced functor $\cC'\otimes \QQ \to \cA$ is also monoidal. 
\end{prop}

An additive category $\cA$ is called pseudo-abelian if every idempotent morphism in $\cA$ admits a kernel. Idempotents in pseudo-abelian categories split:
\begin{prop}[{\cite[Proposition 6.9]{K}}]
Let $e\colon E\to E$ be an idempotent morphism in a pseudo-abelian category. Then
\[ E = \ker(e) \oplus \ker(1-e).\]
Relative to this decomposition the endomorphism $e$ takes the form
\[ e = 0_{\ker(e)} \oplus \id_{\ker(1-e)}. \]
\end{prop}
The Karoubi envelope of an additive category $\cA$ is a pseudo-abelian category $\widetilde{\cC}$ equipped with a full and faithful additive functor $\cA\to \widetilde{\cA}$ satisying the following universal property: if $\cA \to \cB$ is an additive functor from $\cA$ to a pseudo-abelian category $\cB$, then there exists a unique additive functor $\widetilde\cA \to \cB$ making the following diagram commute:
\[ \xymatrix{
\cA \ar[d]\ar[r] & \cB \\
\widetilde\cA \ar@/^0pc/[ru]
}\]
Here is an explicit construction: objects of $\widetilde\cA$ are pairs $(E, e)$ where $e\colon E\to E$ is an idempotent morphism in $\cA$. A morphism $(E, e) \to (E', e')$ is a morphism $f\colon E\to E'$ in $\cA$ such that $fe = e'f = f$. The functor $\cA \to \widetilde \cA$ is given by $X \mapsto (X,\id)$. If $F\colon \cA \to \cB$ is an additive functor to a pseudo-abelian category $\cB$, then $\widetilde\cA \to \cB$ is defined by $(E,e) \mapsto \ker(F(1-e))$. By \cite[Theorem 6.10]{K} this construction does indeed satisfy all the required properties. 
Further, the construction and result go through verbatim upon replacing `additive' by `$\QQ$-linear' everywhere.

Now suppose $\cA$ is also equipped with a monoidal structure. Then $\widetilde\cA$ inherits a monoidal structure (in the evident way) that upgrades $\cA \to\widetilde \cA$ to a monoidal functor. The reader may check:
\begin{prop}
If $\cA \to \cB$ is a monoidal functor to a monoidal pseudo-abelian category $\cB$, then the induced functor $\widetilde\cA \to \cB$ is monoidal.
\end{prop}

\section{Graded (integral) category $\cO$}
Write $\cO_0^{\ZZ}$ for the category of tuples $(\Psi, \Phi, \var, \can)$, where $\Psi,\Phi\in\mathrm{Vect}^{\ZZ}$, and $\var\colon \Phi \to \Psi\langle -1 \rangle$, $\can\colon \Psi\to \Phi\langle -1 \rangle$ are morphisms of graded vector spaces such that 
\[ \var\circ \can =0.\]
\[ \xymatrixrowsep{4pc}\xymatrix{
\Psi \ar@/^2pc/[d]^-{\can}_{+1} \\
\Phi \ar@/^2pc/[u]^-{\var}_{+1}
}\]
The category obtained by omitting gradings is a quiver description of a/any regular integral block of category $\cO$ for $\fsl_2$ (see \cite[\S5.1.1]{St}). The reader may find it pleasurable to also directly relate this category to perverse sheaves on the Riemann sphere with singularity (possibly) at $\infty$ (cf. \cite[\S4]{V}).

Set
\[ \cO^{\ZZ} = \mathrm{Vect}^{\ZZ} \oplus \bigoplus_{k\in \ZZ_{\geq 0}} \cO_{0}^{\ZZ}.\]
Then $\cO^{\ZZ}$ is a graded version of (integral) category $\cO$ for $\fsl_2$ (each copy of $\cO_0^{\ZZ}$ corresponds to a regular block and $\mathrm{Vect}^{\ZZ}$ corresponds to the most singular block).

We have tautologically defined projection functors:
\[ \pr_{-1}\colon \cO^{\ZZ} \to \mathrm{Vect}^{\ZZ} \quad \mbox{and} \quad
\pr_{k}\colon \cO^{\ZZ} \to \cO^{\ZZ}_0, \quad k\in \ZZ_{\geq 0}.\]
Each $\pr_k$ has an evident right inverse $i_k$. 

Define $\pi^{*}\colon \mathrm{Vect}^{\ZZ} \to \cO^{\ZZ}_0$ by
\[ \pi^{*}V = (V\langle 1 \rangle, V\langle 2 \rangle \oplus V, 
\smatrix{0 & \id_{V}}),
\smatrix{\id_{V\langle 1 \rangle} \\ 0 }). 
\]
\[  \xymatrixrowsep{4pc}\xymatrix{
V\langle 1 \rangle \ar@/^2pc/[d]^-{\smatrix{\id_{V\langle 1 \rangle} \\ 0 }}_{+1} \\
V \langle 2 \rangle \oplus V\ar@/^2pc/[u]^-{\smatrix{0 & \id_{V}}}_{+1}
}\]
Define $\pi_*\colon \cO_0^{\ZZ} \to \mathrm{Vect}^{\ZZ}$ by
\[ \pi_*(\Psi, \Phi, \var, \can) = \Phi.\]
The functor $\pi^*$ is left adjoint to $\pi_*$, and $\pi^*\langle -2\rangle$ is right adjoint to $\pi_*$. Furthermore, 
\[ \pi_*\pi^* = \id\oplus \id\langle 2\rangle.\] 
Under the dictionary with category $\cO$, the functor $\pi_*$ corresponds to translation to the wall, and $\pi^*$ corresponds to translation off the wall. Working topologically, $\pi_*$ corresponds to (the Koszul dual) of pushing to a point, and $\pi^*$ corresponds to (the Koszul dual of) pulling back from a point.

Define a functor $F\colon \cO^{\ZZ}\to \cO^{\ZZ}$ by:
\[ F(V) = i_0\pi^*\pr_{-1}(V) \oplus i_{-1}\pi_*\pr_0(V) \oplus i_1\pr_0(V) \bigoplus_{k\geq 1}i_{k-1}\pr_k(V) \oplus i_{k+1}\pr_k(V). \]
Under the translation to category $\cO$, the functor $F$ corresponds to tensoring with $L$ (the tautological $2$-dimensional irreducible representation of $\fsl_2$). 
\begin{prop}\label{actionprop}The functor $F$ yields an action of $\cC$ on $\cO^{\ZZ}$.
\end{prop}

\begin{proof}As $\cO^{\ZZ}$ is abelian, it is pseudo-abelian. Clearly $\cO^{\ZZ}$ is $\QQ$-linear. Hence, it suffices to find morphisms $\id \to F^2$ and $F^2\to \id$ satisfying (i)-(iii) of \S\ref{s:gens}.
Define these morphisms `block by block' as follows. For $V\in \mathrm{Vect}^{\ZZ}$,
\[ F^2(i_{-1}V) = i_1\pi^*V \oplus i_{-1}\pi_*\pi^*V = i_1\pi^*V \oplus i_{-1}V \oplus i_{-1}V\langle 2\rangle.\]
Let $i_{-1}V \to F^2(i_{-1}V)$ and $F^2(i_{-1}V) \to i_{-1}V$ be the evident inclusion and projection maps. Then, for $V\in \cO_0^{\ZZ}$,
\[ F^2(i_0V) = i_0V \oplus i_2V \oplus i_0\pi^*\pi_*V.\]
Define $i_0V \to F^2(i_0V)$ and $F^2(i_0V)\to i_0V$ to be the obvious inclusion and projection.
For $k>0$,
\[ F^2(i_kV) =  i_{k-2}V \oplus i_{k}V \oplus i_{k}V\oplus i_{k+2}V.\]
Let $i_kV\to F^2(i_k V)$ and $F^2(i_kV)\to i_kV$ be the inclusion and projection relative to the second factor.

These morphisms clearly satisfy condition (i) of \S\ref{s:gens}. It is straightforward to verify that they also satisfy conditions (ii)-(iii).
\end{proof}

\section{Crystals}
A crystal will always mean a finite, normal, $\fsl_2$-crystal (see \cite{HK}). This is the data of a finite set $B$ together with maps
 \[ \wt\colon B \to \ZZ, \quad \varepsilon,\phi\to\ZZ, \quad e,f\colon B \to B \sqcup \{0\}, \]
 satisfying the following axioms:
 \begin{enumerate}
 \item For any $b\in B$, $\phi(b) = \varepsilon(b) + \wt(b)$.
 \item Let $b\in B$. If $eb \in B$, then
 \[ \wt(eb) = \wt(b) + 2, \quad \varepsilon(eb) = \varepsilon(b)-1,\quad \phi(eb) = \phi(b)+1.\]
 \item Let $b\in B$. If $fb\in B$, then
 \[ \wt(fb) = \wt(b)-2, \quad \varepsilon(fb) = \varepsilon(b)+1,\quad \phi(fb) = \phi(b)-1.\]
 \item For all $b,b'\in B$ one has $b=eb'$ if and only if $fb=b'$.
 \item For each $b\in B$,
 \[ \varepsilon(b) = \mathrm{max}\{n\,|\, e^n b\neq 0\}, \quad 
 \phi(b) = \mathrm{max}\{n \,|\, f^nb\neq 0\}.\]
 \end{enumerate}
 A morphism of crystals $A\to B$ is a map of sets $\beta\colon A \to B\sqcup \{0\}$ compatible with $e,f,\varepsilon,\phi$ and $\wt$. 
Let $\mathrm{Crystals}$ denote the category of crystals. 
If $B$ and $B'$ are crystals, then the disjoint union $B\sqcup B'$ is a crystal, denoted $B\oplus B'$, in the evident way. This is the coproduct in $\mathrm{Crystals}$.

For $n\in\ZZ_{\geq 0}$ let $\BB(n) = \{v_n, v_{n-2}, \ldots, v_{-n}\}$. Define a crystal structure on $\BB(n)$ by $\wt(v_k) = k$ and 
 \[ ev_k = \begin{cases}
 0 & \mbox{if $k=n$}; \\
 v_{k+2} & \mbox{otherwise}.
 \end{cases} \]

 Let $B,B'$ be crystals. Their tensor product $B\otimes B'$ is defined as follows. As a set $B\otimes B' = B\times B'$. For $a\in B$ and $b\in B'$, write $a\otimes b$ for the corresponding element in $B\times B'$. Then we set
 \begin{align*}
 \wt(a\otimes b) &=\wt(a)+\wt(b), \\
 e(a\otimes b)&= \begin{cases}
 ea\otimes b & \mbox{if $\varepsilon(a) > \phi(b)$}, \\
 a\otimes eb & \mbox{otherwise};
 \end{cases} \\
 f(a\otimes b) &= \begin{cases}
 fa\otimes b &\mbox{if $\varepsilon(a) \geq \phi(b)$}, \\
 a\otimes fb &\mbox{otherwise};
 \end{cases} \\
 \varepsilon(a\otimes b) &= \mathrm{max}\{\varepsilon(b), \varepsilon(a) - \wt(b)\}, \\
 \phi(a\otimes b) &= \mathrm{max}\{\phi(a), \phi(b)+\wt(a)\}.
 \end{align*}
 This endows $\mathrm{Crystals}$ with a monoidal structure. The unit is $\BB(0)$. 

There is a unique isomorphism:
 \[ \BB(1) \otimes \BB(1) \simeq \e \oplus \BB(2) \]
 given by
 \begin{align*}
 \e &\to \BB(1) \otimes \BB(1),  \quad v_{0} \mapsto v_{-1}\otimes v_1, \\
 \BB(2) &\to \BB(1)\otimes \BB(1), \quad v_{2} \mapsto v_1\otimes v_1.
 \end{align*}
 The compositions 
 \[ \BB(1) \to \BB(1)\otimes (\BB(1)\otimes \BB(1)) = (\BB(1)\otimes\BB(1))\otimes \BB(1) \to \BB(1), \]
 \[ \BB(1) \to (\BB(1)\otimes \BB(1))\otimes \BB(1) = \BB(1)\otimes (\BB(1)\otimes \BB(1))\to \BB(1) \]
 are zero.

In view of the above relations, one might naively expect that $\BB(1)\mapsto F$ yields an equivalence between the $\QQ$-linearization of $\mathrm{Crystals}$ and the category $\cC$ of \S\ref{s:gens}. In view of Proposition \ref{actionprop}, this would achieve the goal of this document. Sadly, this is not to be:
\[ |\Hom_{\mathrm{Crystals}}(\BB(1), \BB(1)\oplus \BB(1))| \neq |\Hom_{\mathrm{Crystals}}(\BB(1)\oplus \BB(1), \BB(1))|.\] 
(There are two non-zero morphisms $\BB(1)\to \BB(1) \oplus \BB(1)$, and three non-zero morphisms $\BB(1)\oplus \BB(1) \to \BB(1)$; the `extra' morphism $\BB(1)\oplus\BB(1)\to \BB(1)$ is what one would like to call the `sum' of the two `projections' $\BB(1)\oplus\BB(1)\to\BB(1)$).
Consequently,
\[ |\Hom_{\mathrm{Crystals}}(\BB(1)^{\otimes 3}, \BB(1)^{\otimes 3}) - \{0\}| \neq \dim_{\QQ}(F^3, F^3). \]
This rules out the possibility of the hoped for equivalence.

Regardless, this delinquent behaviour points to a solution. Informally, $\cC$ is a quotient of the $\QQ$-linearization of $\mathrm{Crystals}$. Among other relations, the quotient identifies the sum of the two projections $\BB(1)\oplus \BB(1) \to \BB(1)$ with the `extra' morphism mentioned above. Precisely identifying this quotient is the cause of all the merry contortions of the next section. Before diving into this let's at least formally record:

 \begin{prop}The assignment $F\mapsto \BB(1)$ yields a faithful monoidal functor 
 \[\cC'\to \mathrm{Crystals}.\]
\end{prop}

\begin{proof}
Clearly, we have a monoidal functor $\cC'\to\mathrm{Crystals}$. Faithfulness follows from the string diagram description of morphisms in $\cC'$.
\end{proof}

\section{$\mathrm{Crystals}_{\QQ}$}
Let $\cC''$ be the following category:
\begin{itemize}
\item objects: $\bigoplus_{n\in\ZZ_{\geq 0}} \BB(n) \boxtimes V_n$, where $V_n$ is a finite set;
\item morphisms: a morphism $\bigoplus_{n\in \ZZ_{\geq 0}}\BB(n) \boxtimes V_n \to \bigoplus_{n\in \ZZ_{\geq 0}} \BB(n) \boxtimes W_n$ is the datum of a map of sets $V_n \to W_n \sqcup \{0\}$ for each $n$.
\end{itemize}
\begin{prop}\label{combdescription}
The category $\mathrm{Crystals}$ is equivalent to $\cC''$.
\end{prop}

\begin{proof}For a crystal $B$, let $\Hom^{\neq 0}(\BB(n), B)$ be the set of non-zero morphisms $\BB(n)\to B$. Define a functor $G\colon\mathrm{Crystals} \to \cC''$ by
\[ G(B) = \bigoplus_{n\in \ZZ_{\geq 0}} \BB(n) \boxtimes \Hom^{\neq 0}(\BB(n), B). \]
For a morphism of crystals $\phi\colon B\to B'$, define the family of maps $G(\phi)$ as follows. Let $f\in \Hom^{\neq 0}(\BB(n), B)$.
Then $G(\phi)$ maps $f$ to $\phi\circ f$. This functor is full and faithful. Further, its essential image is $\cC''$.
\end{proof}
Identify $\mathrm{Crystals}$ with $\cC''$. Define finite sets $V_{ij}^k$ by:
\begin{equation}\label{finsets} \BB(i) \otimes \BB(j) = \bigoplus_k \BB(k) \boxtimes V_{ij}^k. \end{equation}
Then
\begin{align*}
(\BB(i) \otimes \BB(j)) \otimes \BB(k) &= \bigoplus_{l} \BB(l) \boxtimes \left( \bigsqcup_{\alpha} V_{ij}^{\alpha} \times V_{\alpha k}^l \right), \\
\BB(i) \otimes (\BB(j)\otimes \BB(k)) &= \bigoplus_l \BB(l) \boxtimes
\left( \bigsqcup_{\beta} V_{i\beta}^l \times V_{jk}^{\beta} \right).
\end{align*}
The (strict) associativity of $\mathrm{Crystals}$ yields bijections:
\begin{equation}\label{associ} \phi^{l}_{ijk} \colon\bigsqcup_{\alpha} V_{ij}^{\alpha} \times V_{\alpha k}^l \mapright{\sim}
\bigsqcup_{\beta} V_{i\beta}^l \times V_{jk}^{\beta}.\end{equation}
These bijections also satisfy some equations (born from the pentagon axiom) which we need not explicitly formulate.

In effect, $\cC''$ yields a more linear algebraic description of $\mathrm{Crystals}$ at the cost of losing strict associativity.
Now define $\mathrm{Crystals}_{\QQ}$ as follows:
\begin{itemize}
\item objects: $\bigoplus_{n\in\ZZ_{\geq 0}} \BB(n) \boxtimes V_n$, where $V_n$ is a finite set;
\item morphisms: a morphism $\bigoplus_{n\in \ZZ_{\geq 0}}\BB(n) \boxtimes V_n \to \bigoplus_{n\in \ZZ_{\geq 0}} \BB(n) \boxtimes W_n$ is the datum of a linear map 
\[ \QQ\mathrm{-span}\, V_n \to \QQ\mathrm{-span}\, W_n, \]
for each $n$.
\end{itemize}
$\mathrm{Crystals}_{\QQ}$ is $\QQ$-linear and pseudo-abelian. 
We have an evident faithful functor
\begin{equation}\label{ft} \mathrm{Crystals} \mapright{\sim} \cC'' \to \mathrm{Crystals}_{\QQ}. \end{equation}
Note:
\[ \dim_{\QQ}\Hom_{\mathrm{Crystals_{\QQ}}}(\BB(1)^{\otimes m}, \BB(1)^{\otimes n}) = \dim_{\CC}\Hom_{\mathrm{Rep}(\fsl_2)}(L^{\otimes m}, L^{\otimes n}). \]
Further, the formula \eqref{finsets} yields a bifunctor on $\mathrm{Crystals}_{\QQ}$. We define an associativity constraint using the data \eqref{associ}. This upgrades $\mathrm{Crystals} \to\mathrm{Crystals}_{\QQ}$ to a monoidal functor.

\begin{prop}$\mathrm{Crystals}_{\QQ}$ is monoidally equivalent to $\cC$.
\end{prop}

\begin{proof}As $\mathrm{Crystals}_{\QQ}$ is $\QQ$-linear and pseudo-abelian, the faithful monoidal functor $\cC' \to \mathrm{Crystals}\to \mathrm{Crystals}_{\QQ}$ induces a faithful monoidal functor 
\[ \cC \to \mathrm{Crystals}_{\QQ}.\]
As
\begin{align*}
\dim_{\QQ} \Hom_{\cC}(F^m, F^n) &= \dim_{\CC} \Hom_{\mathrm{Rep}(\fsl_2)}(L^{\otimes m}, L^{\otimes n}) \\
&= \dim_{\QQ} \Hom_{\mathrm{Crystals}_{\QQ}}(\BB(1)^{\otimes m}, \BB(1)^{\otimes n}), 
\end{align*}
the faithfulness implies that the functor is also full. As $\BB(n)$ occurs as a direct summand of $\BB(1)^{\otimes n}$, and every object in $\mathrm{Crystals}_{\QQ}$ is a direct sum of $\BB(k)$s, it follows that $\cC\to \mathrm{Crystals}_{\QQ}$ is essentially surjective.
\end{proof}
In view of Proposition \ref{actionprop} and \eqref{ft} this yields an action of $\mathrm{Crystals}$ on $\cO^{\ZZ}$.


\begin{thebibliography}{99}
\bibitem[BGS]{BGS} {\sc A. Beilinson, V. Ginzburg, W. Soergel}, {\em Koszul duality patterns in Representation Theory}, Journal of A.M.S. \textbf{9} (1996), 473-526.
\bibitem[BGG]{BGG} {\sc J. Bernstein, I. M. Gelfand, S. I. Gelfand}, {\em A certain category of $\fg$-modules}, Functional Analysis and its Applications \textbf{10}, no. 2 (1976), 1-8.
\bibitem[BS]{BS} {\sc J. Brundan, C. Stroppel}, {\em Highest weight categories arising from Khovanov's diagramm algebra III: category $\cO$}, Represent. Theory \textbf{15} (2011), 170-243.
\bibitem[FK]{FK} {\sc I. Frenkel, M. Khovanov}, {\em Canonical bases in tensor products and graphical calculus for $U_q(\fsl_2)$}, Duke Math. J. \textbf{87}, 3 (1997), 409-480.
\bibitem[HK]{HK} {\sc J. Hong, S-J. Kang}, {\sl Introduction to quantum groups and crystal bases}, Graduate Studies in Mathematic \textbf{42}, A.M.S., Providence, RI (2002).
\bibitem[K]{K} {\sc M. Karoubi}, {\sl K-theory: an introduction}, Grundlehren der math. Wiss. \textbf{226}, Springer-Verlag (1978).
\bibitem[S]{S} {\sc J.G. Saxe}, {\em The Blind Men and the Elephant}.
\bibitem[So]{So} {\sc W. Soergel}, {\em Book review}, Bulletin of the A.M.S., Vol 47, Number 2, April 2010, 367-371.
\bibitem[St1]{St1} {\sc C. Stroppel}, {\em Category $\cO$: gradings and translation functors}, J. Algebra \textbf{268}, no. 1 (2003), 301-326.
\bibitem[St2]{St} {\sc C. Stroppel}, {\em Category $\cO$: quivers and endomorphism rings of projectives}, Represent. Theory \textbf{7} (2003), 322-345.
\bibitem[V]{V} {\sc J-L. Verdier}, {\em Extension of a perverse sheaf over a closed subspace}, Ast\'erisque \textbf{130} (1985), 210-217.
\end{thebibliography}
\end{document}